\documentclass[letterpaper,10pt,conference]{ieeeconf}

\usepackage{mathtools}
\usepackage{amssymb, amsmath, bm}
\usepackage[thmmarks, amsmath]{ntheorem}
\usepackage{xcolor}
\usepackage{balance} 

\usepackage[labelformat=simple]{subcaption}
\usepackage{hyperref}
\let\labelindent\relax
\usepackage{enumitem}
\newcommand\numberthis{\addtocounter{equation}{1}\tag{\theequation}}

\usepackage{chngcntr}
\makeatletter 
\newcounter{parentassumption} 

\newenvironment{assumptionset}{%
  \refstepcounter{assumption}
  \setcounter{parentassumption}{\value{assumption}}
  \protected@edef\theparentassumption{\theassumption}
  \setcounter{assumption}{0}
  \renewcommand{\theassumption}{\theparentassumption\alph{assumption}}
}{%
  \setcounter{assumption}{\value{parentassumption}}
}
\makeatother 

\allowdisplaybreaks

\usepackage[style=ieee,dashed=false]{biblatex}
\IEEEoverridecommandlockouts                             
\title{\fontsize{15.6}{19.5}\selectfont  \bf
Huber-based Robust System Identification with Near-Optimal Guarantees Across Independent and Adversarial Regimes}

\author{Jihun Kim and Javad Lavaei
\thanks{This work was supported by the U. S. Army Research Laboratory and the U. S. Army Research Office under Grant W911NF2010219, Office of Naval Research under Grant N000142412673, and NSF. Jihun Kim and Javad Lavaei are with the Department of Industrial Engineering and Operations Research, University of California, Berkeley. 
Emails:  {\tt\footnotesize \{jihun.kim, lavaei\}@berkeley.edu}}
}

\theoremstyle{plain} 
\theoremnumbering{arabic} 
\theoremseparator{.} 
\newtheorem{theorem}{Theorem}
\newtheorem{lemma}{Lemma}

\newtheorem{proposition}{Proposition}

\theoremstyle{plain} 
\theoremheaderfont{\bfseries\itshape} 
\theoremnumbering{arabic} 
\theoremseparator{:}
\newtheorem{scenario}{Scenario}

\theoremstyle{plain} 
\theoremheaderfont{\itshape} 
\theorembodyfont{\itshape} 
\theoremnumbering{arabic} 
\theoremseparator{:} 

\newtheorem{assumption}{Assumption}

\theoremstyle{plain}
\theoremheaderfont{\itshape}  
\theorembodyfont{\normalfont} 
\theoremseparator{:}          

\newtheorem{remark}{Remark}

\allowdisplaybreaks

\begin{document}

\maketitle 
\thispagestyle{empty}
\pagestyle{empty}

\begin{abstract}
Dynamical systems can confront one of two extreme types of disturbances: persistent zero-mean independent noise, and sparse  nonzero-mean adversarial attacks, depending on the specific  scenario being modeled. While mean-based estimators like least-squares are well-suited for the former, a median-based approach such as the $\ell_1$-norm estimator is required for the latter. In this paper, we propose a Huber-based estimator, characterized by a threshold constant $\mu$, to identify the governing matrix of a linearly parameterized nonlinear system from a single trajectory of length $T$. This formulation bridges the gap between mean- and median-based estimation, achieving provably robust error in both extreme disturbance scenarios under mild assumptions. In particular, for persistent zero-mean noise with a positive probability density around zero, the proposed estimator achieves an $\mathcal{O}(1/\sqrt{T})$ error rate if the disturbance is symmetric or  the basis functions are linear. For arbitrary nonzero-mean attacks that occur at each time with probability smaller than 0.5, the error is bounded by $\mathcal{O}(\mu)$. We validate our theoretical results with experiments illustrating that integrating our approach into frameworks like SINDy yields robust identification of discrete-time systems.
\end{abstract}

\section{Introduction}\label{intro}

Modern engineering systems are inherently complex and their governing dynamics are frequently partially or fully unknown. System identification is the procedure of learning underlying models based on the state trajectory generated by the system, providing the foundation to design robust and reliable control algorithms \cite{ljung1998sys}. 
For large-scale infrastructure such as power systems, however, collecting data via forcing state resets  often incurs massive operational downtime \cite{ghodeswar2025quantifying}. Similarly, human patients in clinical control  cannot be physiologically reset \cite{allam2021analyzing}. Thus, practical system identification in many real-world applications necessitates learning from a single trajectory. We formulate this task as a parameter estimation problem and consider a structured discrete-time dynamical system generating a sequence of the form
\begin{align}\label{linearlyparam}
    x_{t+1} = \bar A \phi(x_t) +  w_t,\quad t=0,\dots, T-1,
\end{align}
where $x_t\in\mathbb{R}^n$ is the state, $w_t\in\mathbb{R}^n$ is the disturbance at time $t$, and $T$ denotes the trajectory length. The system dynamics are governed by an unknown target matrix
$\bar A \in\mathbb{R}^{n\times m}$—with rows $\bar a_1^T,\dots, \bar a_n^T$—and known, potentially nonlinear basis functions  $\phi:\mathbb{R}^n\to\mathbb{R}^m$ chosen by the system analyst. Given a single trajectory $\{x_t\}_{t=0}^{T}$, our objective is to accurately estimate the unknown values in $\bar A$.

The robustness of this identification problem is often challenged by the nature of disturbances $\{w_t\}_{t=0}^{T-1}$. 
In recent literature, two distinct disturbance regimes are often considered: (a) \textit{\textbf{persistent zero-mean independent noise}}, and (b) \textit{\textbf{sparse nonzero-mean adversarial attack}}. 
The former typically arises as natural fluctuations in physical systems, commonly modeled as a non-adversarial (white) noise process present at every time \cite{freidlin2012random}. In contrast, the latter appears in security and fault-diagnosis settings, where an adversary intermittently but strategically corrupts the system \cite{teixeira2015secure}.

To  address these two extreme scenarios, recent literature has adopted two primary classes of estimators. 
For zero-mean independent noise, mean-based estimators such as the least-squares method ensure that persistent disturbances are averaged out.  Conversely, for zero-median sparse attacks, median-based estimators such as the $\ell_1$-norm estimator filter out any disturbances of adversarial nature. These two estimators correspond to the following optimization problems:
\begin{align}
    &\min_{A\in\mathbb{R}^{n\times m}} \sum_{t=0}^{T-1} \|x_{t+1} - A\phi(x_t)\|_2^2 \tag{Least-squares}\label{ls} \\&\min_{A\in\mathbb{R}^{n\times m}} \sum_{t=0}^{T-1} \|x_{t+1} - A\phi(x_t)\|_1 \tag{$\ell_1$-norm estimator}  \label{l1}
\end{align}

It is well-established that the least-squares method achieves an $\mathcal{O}(1/\sqrt{T})$ error rate under persistent zero-mean independent noise with high probability after a certain finite time complexity \cite{simchowitz2018learning, jedra2020finite}. As shown in  \cite{simchowitz2018learning}, this rate is indeed optimal for such a persistent noise regime. 
Conversely, under sparse nonzero-mean adversarial attacks, where the median of attacks is zero provided that the attack probability is smaller than 0.5, the $\ell_1$-norm estimator shows a  fundamentally different behavior: it achieves exact recovery (zero error) with high probability within  finite time  \cite{kim2024prevailing}.

Despite their individual efficacy, both approaches face a critical blind spot. 
The former approach is limited to zero-mean noise and fails against adversarial attacks, since biased disturbances do not average to zero. The latter approach, while robust to sparse attacks, cannot overcome persistent noise, since it  requires the disturbance to be exactly zero with probability greater than 0.5. This leads to the central challenge of system identification:  \textit{in practice, the underlying nature of the disturbances is unknown in advance}.

\textbf{Contribution. } In this paper, we provide the robust system identification framework  via the \textbf{\textit{Huber estimator}}, using the Huber loss \cite{huber1964robust} defined by a threshold constant $\mu$. We establish the first theoretical guarantees that the Huber estimator is universally effective ``near the best of both worlds'':
\begin{enumerate}[leftmargin=0.5cm]
    \item The Huber estimator recovers the optimal $\mathcal{O}(1/\sqrt{T})$ error rate under persistent zero-mean independent noise, when the noise has a positive probability density around zero. 
    \item The Huber estimator ensures that the estimation error is bounded by a constant $\mathcal{O}(\mu)$ error under sparse nonzero-mean adversarial attack. 
\end{enumerate}

Although \cite{kumar2025machine} recently demonstrated the empirical robustness of  the Huber estimator when applied to neural networks under distinct disturbance scenarios,  theoretical foundations have not yet been developed. We bridge this gap by providing a rigorous analysis for the estimation of $\bar A$ in \eqref{linearlyparam}. Our framework incorporates a SINDy-type structure \cite{brunton2016discovering} via sufficiently expressive nonlinear basis functions $\phi(x)$ and a sparse target matrix $\bar A$. Although the original SINDy approach relies on  least-squares suited only for zero-mean noise, our work provides theoretical guarantees to handle both extreme disturbance regimes via the Huber estimator.

\textbf{Outline. } In Section \ref{sec:probform}, we outline the relevant assumptions for each disturbance scenario. Section \ref{sec:huberbased} formalizes the Huber estimator. Section \ref{sec:mainresults} presents the two main theoretical results under these scenarios, and Section \ref{sec:numerical} provides numerical validation of our claims. The detailed proof techniques for each theorem are developed in Sections \ref{sec:proofthm1} and \ref{sec:proofthm2}. Finally, Section \ref{sec:conclusion} provides concluding remarks.

\textbf{Notation. } For a vector $x$, $x^T$  is a transpose of a vector and  $x^{(i)}$ denotes its $i^\text{th}$ entry. The notation $\|\cdot \|_2$ denotes the $\ell_2$-norm for vectors and the operator norm for matrices, while $\|\cdot \|_1$ denotes the $\ell_1$-norm for vectors. 
The notation $I$ denotes the identity matrix. Let $\bm{\sigma}(\cdot)$ denote the sigma-algebra. 
$\mathbb{E}[\cdot]$ denotes expectation and $\mathbb{P}(\cdot)$ denotes probability. The notations $\mathcal{O}(\cdot)$ and $\Omega(\cdot)$ indicate an upper and a lower bound up to a positive constant, respectively. A distribution $w$ is symmetric if $w$ and $-w$ are identically distributed.

\section{Problem Formulation and Core Assumptions}\label{sec:probform}

In this section, we state the assumptions and scenarios required to establish theoretical guarantees for the Huber estimator. The first two assumptions ensure that the trajectories generated by the underlying true system do not diverge.

\begin{assumption}[System Stability]\label{as:stability}
Let $\rho$ denote the spectral norm $\|\bar A\|_2$. Let $L$ be a Lipschitz constant for $\phi(\cdot)$; \textit{i.e.}, $\|\phi(x)-\phi(\tilde x )\|_2 \le L\|x-\tilde x\|_2$ for all $x,\tilde x \in\mathbb{R}^n$. Moreover, $\phi(0)= 0$.  The stability condition is $\rho L< 1$.
\end{assumption}

\begin{assumption}[Sub-Gaussian Disturbance]\label{as:subg}
Define $\mathcal{F}_t = \bm{\sigma}\{x_0, \dots, x_t\}.$
Assume that all $w_t$ and $x_0$ are sub-Gaussian vectors\footnote{The notion of sub-Gaussian variables is introduced in Section 2.6, \cite{vershynin2025high}. A variable $x$ is sub-Gaussian if its $\psi_2$-norm $\|x\|_{\psi_2}= \inf\{k>0: \mathbb{E}[\exp(x^2/k^2)] \le e\}$ is finite. For example, every bounded variable is sub-Gaussian.
Furthermore, a vector $X \in \mathbb{R}^n$ is defined as sub-Gaussian with $\psi_2$-norm $\sigma$ if the scalar projection $u^T X$ is sub-Gaussian with $\psi_2$-norm of at most $\sigma$ for all $u \in \mathbb{R}^n$ such that $\|u\|_2 = 1$. } (not necessarily zero-mean or independent); \textit{i.e.}, there exists $\sigma > 0$ such that $\|x_0^{(i)}\|_{\psi_2}\le \sigma$ and $\|w_t^{(i)}\,|\,\mathcal{F}_t\|_{\psi_2} \le \sigma$ for every $t \ge 0$ and $i\in\{1,\dots,n\}$. 
\end{assumption}

\begin{remark}\label{subgnorm}
The assumption $\phi(0) = 0$ is made for simplicity and can be relaxed for any Lipschitz $\phi$ by introducing an additive constant to the upcoming theorems.
To leverage Lipschitz continuity, it is useful to  study the $\psi_2$-norm of the $\ell_2$-norm for sub-Gaussian variables.
 Assumption \ref{as:subg} implies that $\| \|x_0\|_2 \|_{\psi_2}$ and $\| \|w_t\|_2 \|_{\psi_2}$ are bounded by $\sqrt{n}\sigma$, which is tight for independent coordinates; \textit{e.g.}, when $w_t$ follows a Gaussian  $\mathcal{N}(0, \sigma^2 I)$, its $\psi_2$-norm scales with $\sigma$ and 
 $\mathbb{E}[\|w_t\|_2^2]=\sum_{i=1}^n \mathbb{E}[(w_t^{(i)})^2] = n\sigma^2$, which means $\|w_t\|_2$ concentrates around its expected value of roughly $\sqrt{n}\sigma$. 
\end{remark}

We now consider two extreme scenarios of disturbances $\{w_t\}_{t=0}^{T-1}$: \textit{persistent zero-mean independent noise} and \textit{sparse nonzero-mean adversarial attack}.  

\begin{scenario}[Persistent Zero-mean Independent Noise]\label{as:white}
     $w_t$ is an independent, zero-mean process for $t\ge0$. Moreover, $\mathbb{E}[x_0]=0$. 
\end{scenario}

\begin{scenario}[Sparse Nonzero-mean Adversarial Attack]\label{as:prob}
    $w_t$ is an attack at time $t$ with probability $p < 0.5$, conditioned on $\mathcal{F}_t= \bm{\sigma}\{x_0, \dots, x_t\}$. Formally, there exists a sequence $(\xi_t)_{t = 0}^{T-1}$ of independent $\mathrm{Bernoulli}(p)$ variables
    such that
$\{\xi_t = 0\} \subseteq \{w_t = 0\}$
for all $t\ge0$. Under this restriction on attack times, $w_t$ can be chosen arbitrarily by an adversary with access to $\mathcal{F}_t$ at every attack time $t\ge 0$.
\end{scenario}

Scenario \ref{as:white} specifies that the system is under independent zero-mean noise at every time, while Scenario \ref{as:prob} states that the system is under adversarial attack at each time with probability smaller than $0.5$.
The next set of assumptions formalizes the sufficient expected excitation for each regime. 

\begin{assumptionset}

\begin{assumption}[Expected Excitation]\label{as:excitation}
There exists $\lambda>0$ such that $\mathbb{E}[\phi(x_t)\phi(x_t)^T\mid  \mathcal{F}_{t-1}] \succeq \lambda^2 I$ for all $t=1,\dots, T$, meaning that $\phi(x_t)$ covers entire space in $\mathbb{R}^m$ in expectation.    
\end{assumption}

\begin{assumption}[Expected Excitation under Attack]\label{as:excitation_attack}
There exists $\lambda>0$ such that $\mathbb{E}[\phi(x_t)\phi(x_t)^T\mid \xi_{t-1}=1, \mathcal{F}_{t-1}] \succeq \lambda^2 I$ for all $t=1,\dots, T$, meaning that $\phi(x_t)$ covers entire space in $\mathbb{R}^m$ in expectation, whenever attack happens.
\end{assumption}
\end{assumptionset}

Throughout the paper, we will establish theoretical guarantees on the estimation error of the Huber estimator across both disturbance scenarios.

\section{Huber-based System Identification}\label{sec:huberbased}

In this section, we formalize the Huber estimator and present the underlying intuition that motivates its robustness.
Given $\mu>0$, consider
\begin{align}\tag{Huber estimator}  \label{hubermin}
    \min_{\{a_i\}_{i=1}^n} \sum_{t=0}^{T-1} \sum_{i=1}^n H_{\mu} (x_{t+1}^{(i)} - a_i^T \phi(x_t)), 
\end{align}
where 
\begin{align}\label{huberfnc}
    H_{\mu} (z) = \begin{dcases*}
        \frac{1}{2}z^2 & if $|z|\le \mu$, \\
         \mu |z| -\frac{1}{2}\mu^2 & if $|z|>\mu$.
    \end{dcases*}
\end{align}
Note that the term $x_{t+1}^{(i)} - a_i^T \phi(x_t)$ is the $i^\text{th}$ entry of the residual $x_{t+1}- A\phi(x_t)$, where $a_i^T$ denotes the $i^\text{th}$ row of $A$.

The Huber loss \eqref{huberfnc} is convex and acts as a quadratic penalty for small arguments and a linear penalty for large ones. 
The following proposition formalizes how this dual behavior allows the Huber estimator to bridge between the least-squares and the $\ell_1$-norm estimators.

\begin{proposition}
    The problem \eqref{hubermin} is equivalent to the problem 
    \begin{align}\label{lasso}
    \min_{A, \{v_t\}_{t\ge 0}} \sum_{t=0}^{T-1} \frac{1}{2}\|x_{t+1}  - A\phi(x_t)-v_t\|_2^2 + \mu \|v_t\|_1,
\end{align}
where $A$ is a matrix whose rows are $a_i^T$ for $i\in\{1,\dots,n\}$.
\end{proposition}

\begin{proof}
   The joint minimization with respect to $A$ and $\{v_t\}_{t\ge 0}$ is equivalent to  minimizing first with respect to $A$, and then over $\{v_t\}_{t\ge 0}$. Thus, it suffices to show that 
    \begin{equation}\label{eqhuber}
    \begin{split}
       & \min_{\{v_t\}_{t\ge 0}} \sum_{t=0}^{T-1} \frac{1}{2}\|x_{t+1}  - A\phi(x_t)-v_t\|_2^2 + \mu \|v_t\|_1 \\&\hspace{30mm}= \sum_{t=0}^{T-1} \sum_{i=1}^n H_{\mu} (x_{t+1}^{(i)} - a_i^T\phi(x_t))
       \end{split}
    \end{equation}
    for all $A$. Since the left-hand side term is convex in $\{v_t\}_{t\ge 0}$ and can be decoupled along the time $t$ as well as the coordinate $i$, the KKT optimality conditions imply that
    \begin{align*}
       0\in   -(x_{t+1}^{(i)} - a_i^T \phi(x_t)-v_t^{(i)}) + \mu \partial |v_t^{(i)}|, \quad t\ge 0
    \end{align*}
    for every $i\in\{1,\dots,n\}$, where $\partial$ denotes the subderivative.
We consider two cases based on the KKT conditions.

\textbf{Case 1}: $|x_{t+1}^{(i)} - a_i^T \phi(x_t)| \le \mu$. In this case, $v_t^{(i)} = 0$. To see why, note that if $v_t^{(i)} > 0$, the KKT conditions would require $x_{t+1}^{(i)} - a_i^T \phi(x_t) - v_t^{(i)} = \mu$, which incurs a contradiction. The argument for $v_t^{(i)} < 0$ follows similarly.

\textbf{Case 2}: $|x_{t+1}^{(i)} - a_i^T \phi(x_t)| > \mu$. Here, $v_t^{(i)} \ne 0$. If $x_{t+1}^{(i)} - a_i^T \phi(x_t) > \mu$, the KKT conditions imply $v_t^{(i)} = x_{t+1}^{(i)} - a_i^T \phi(x_t) - \mu>0$. Alternatively, if $x_{t+1}^{(i)} - a_i^T \phi(x_t) < -\mu$, the conditions imply $v_t^{(i)} = x_{t+1}^{(i)} - a_i^T \phi(x_t) + \mu<0$.

    Substituting the obtained $v_t^{(i)}$ for every $t$ and $i$ back into the left-hand side of \eqref{eqhuber} yields the right-hand side. This completes the proof.
\end{proof}

\begin{remark}\label{propremark}
 The above proposition introduces an alternative convex formulation of the Huber estimator, consisting of the least-squares term $\sum_t\frac{1}{2} \|x_{t+1} - A\phi(x_t)-v_t\|_2^2$ and the $\ell_1$-regularization term $\mu\sum_t \|v_t\|_1$, which implies that the Huber estimator serves as a middle ground between the least-squares and the $\ell_1$-norm estimator. The value of $\mu$ dictates how close the Huber estimator is to either one of the two aforementioned estimators; 
 when $\mu$ is large, the heavy penalty forces $v_t$ to zero, reducing the objective nearly to the  least-squares estimator given in \eqref{ls}.
 Conversely, when $\mu$ is small,  $v_t$ is forced to approach the residual $x_{t+1} - A\phi(x_t)$ to minimize the quadratic loss, effectively recovering the $\ell_1$-norm estimator defined in \eqref{l1}.  Note that $\mu>0$ is strictly required to ensure that \eqref{lasso} is well-posed; at $\mu=0$, any $A$ is optimal by choosing  $v_t = x_{t+1} - A\phi(x_t)$.
\end{remark}

In the next section, we will provide  theoretical guarantees of the error rates provided by the Huber estimator.

\section{Main Results: Near the Best of Both Worlds}\label{sec:mainresults}

In this section, we provide the main results for the Huber estimator under two extreme disturbance regimes. Technical proofs are all deferred to Sections \ref{sec:proofthm1} and \ref{sec:proofthm2}. The first main theorem shows that when the independent mean-zero noise is persistently injected into the system, the Huber estimator achieves the optimal $\mathcal{O}(1/\sqrt{T})$ under an additional mild assumption on the noise. We present the assumption and the theorem below.

\begin{assumption}\label{as:mildinnoise}
      There exists a universal value $q>0$ such that 
      $\mathbb{P}(|w_t^{(i)}|\le \frac{\mu}{2} ) \ge q $ holds for every $t\ge0$ and $i\in\{1,\dots,n\}$.
\end{assumption}

\begin{theorem}\label{noisescenario}
 Consider Scenario \ref{as:white} and suppose that Assumptions \ref{as:stability}, \ref{as:subg},   \ref{as:excitation}, and \ref{as:mildinnoise} hold.  Let $\hat a_1, \dots, \hat a_n$ be a minimizer to \eqref{hubermin}, and let $\bar a_1^T, \dots, \bar a_n^T$ be each row of $\bar A$. 
  Given $\delta\in(0,1)$, when 
    \begin{itemize}
       \item $w_t$ has a symmetric distribution for all $t\ge 0$, or
        \item $\phi(x) = \bar B x$ for some $\bar B\in\mathbb{R}^{m\times n}$, where $\|\bar B\|_2\le L$,
    \end{itemize}
with
 \begin{align*}
    & T = {\Omega} \biggr( \frac{n^2 L^4 \sigma^4 }{ q^2\lambda^4 (1-\rho L)^4} \log^2\left(\frac{mn}{\delta}\right)\log\left( \frac{nL\sigma }{q\lambda (1-\rho L) \delta} \right) \biggr),\numberthis\label{timeboundfinalfinal}
 \end{align*}
 it holds that
\begin{align*}
   & \|\bar a_i - \hat a_i\|_2 = \mathcal{O}\biggr(\frac{\mu \sqrt{n} L \sigma}{\sqrt{T}q\lambda^2 (1-\rho L)}\log\Bigr(\frac{n}{\delta}\Bigr)\biggr), \\&\hspace{55mm}\forall i\in\{1,\dots, n\}
\end{align*}
with probability at least $1-\delta$.
\end{theorem}

\begin{remark}
   This theorem demonstrates that there exists a finite time complexity beyond which the estimation error is bounded by $\mathcal{O}(\sqrt{n}\log n/\sqrt{T})$. 
We now discuss the conditions on the theorem beyond standard independent zero-mean noise (Scenario \ref{as:white}); we require (a)  a  positive probability density around zero (Assumption \ref{as:mildinnoise}), and (b) the symmetric disturbance unless the system is linear. 
In engineering practice, process noise naturally aligns with this condition: digital quantization errors are often modeled as zero-centered uniform distribution \cite{widrow2008quantization}, while thermal noise is driven by the aggregation of countless independent electron movements, which converges to a zero-mean Gaussian via the Central Limit Theorem \cite{vasilescu2005electronic}. Crucially, the noise in both of these standard scenarios is perfectly symmetric.
\end{remark}

\begin{remark}
  From an analytical perspective, Assumption \ref{as:mildinnoise} is required since the Huber estimator achieves sufficient empirical excitation exclusively via samples with small estimated errors $x_{t+1}^{(i)}-a_i^T\phi(x_t)$; the Huber loss \eqref{huberfnc} only preserves the excitation for a quadratic penalty, whereas least-squares method relies on sufficient empirical excitation averaged over all time steps. Furthermore, symmetry of disturbances is required since applying the Huber penalty effectively truncates the estimated sample error  at $[-\mu, \mu]$, which introduces a bias if the underlying zero-mean disturbance is asymmetric.  This symmetry requirement is completely circumvented in linear systems; since linear basis functions maintain the zero-mean nature of the states $x_t$, we can derive the optimal $\mathcal{O}(1/\sqrt{T})$ error rate regardless of disturbance symmetry.
\end{remark}

The second main theorem shows that the adversarial attack scenario is solved by the Huber estimator with the error bounded by $\mathcal{O}(\mu)$.

\begin{theorem}\label{attackscenario}
     Consider Scenario \ref{as:prob} and suppose that Assumptions \ref{as:stability}, \ref{as:subg}, and \ref{as:excitation_attack} hold.  Let $\hat a_1, \dots, \hat a_n$ be a minimizer to \eqref{hubermin}, and let $\bar a_1^T, \dots, \bar a_n^T$ be each row of $\bar A$. Given $\delta\in(0,1)$, when 
\begin{align}\label{timeboundattack}
      &  T=\Omega\biggr(\frac{(\sqrt{n}L\sigma)^4}{\lambda^4p(1-2p) } \\\nonumber&\hspace{5mm} \times \max\biggr\{   \frac{(\sqrt{n}L\sigma)^{6}\log^2(\frac{n}{\delta})}{(1-2p)\lambda^{6}(1-\rho L)^2 } , m\log\Bigr( \frac{nL\sigma }{\lambda(1-\rho L)}\Bigr) \biggr\}  \biggr),
    \end{align}
it holds that
\begin{align*}
        \|\bar a_i - \hat a_i\|_2 = \mathcal{O}\Bigr(\frac{\mu n^2 L^4\sigma^4 }{p(1-2p)\lambda^5} \Bigr), \quad \forall i\in\{1,\dots, n\}
    \end{align*}
with probability at least $1-\delta$.
\end{theorem}

\begin{remark}
This theorem demonstrates that after a finite time complexity, the Huber estimator shows a bounded error $\mathcal{O}(\mu)$. The proof uses the fact that the $\ell_1$-norm estimator indeed achieves a zero error within  finite time  under Scenario \ref{as:prob} \cite{kim2024prevailing}; subsequently, we use the boundedness of the difference between the $\ell_1$-norm loss and the Huber loss. 
\end{remark}

Theorems \ref{noisescenario} and \ref{attackscenario} elucidate the trade-offs involved in selecting $\mu$ across two extreme disturbance regimes. For sparse nonzero-mean attacks, minimizing $\mu$ tightly bounds the estimation error, which is natural since small $\mu$ corresponds to the $\ell_1$-norm estimator as noted in Remark \ref{propremark}. 
Conversely, under persistent zero-mean noise, $\mu$ must exceed a strictly positive threshold; Assumption \ref{as:mildinnoise} requires a positive probability density across $[-\frac{\mu}{2}, \frac{\mu}{2}]$, which is a condition impossible to satisfy with $\mu=0$ for any absolutely continuous distribution. 
Importantly, however, we do not necessarily require $\mu$ to be arbitrarily large (which corresponds to the least-squares as noted in Remark \ref{propremark}), but merely needs to satisfy the assumption to achieve the optimal $\mathcal{O}(1/\sqrt{T})$ rate. Since increasing $\mu$ beyond a certain threshold does not provide benefit for persistent zero-mean noise,  our theoretical results suggest a clear tuning strategy: $\mu$ should be set to the minimal value that satisfies Assumption \ref{as:mildinnoise} for the Huber estimator to provide the optimal defense against both extreme disturbance scenarios.

\section{Numerical Experiments}\label{sec:numerical}

In this section, we provide experimental validations tested on the discrete-time dynamical systems. For pictorial illustration, we consider 
three-dimensional $x=[x_1,x_2,x_3]^T\in\mathbb{R}^3$.  
\begin{align*}
  \phi(x)= [ x_1, 
        x_2, 
        x_3, 
        &x_1 x_2, 
        x_2  x_3, 
        x_3  x_1, \\&
        x_1^2, 
        x_2^2, 
        x_3^2, 
        \sin(x_1 x_2), 
        \cos(x_3)]^T\in \mathbb{R}^{11}
\end{align*}
with
\begin{footnotesize} 
\begin{equation*}\setlength{\arraycolsep}{3.2pt}
    \bar A = \left[ \begin{array}{ccccccccccc}
0.8 & -0.5 & 0 & 0 & 0.4 & 0 & 0 & 0 & 0 & 0.1 & 0 \\
0.5 & 0.8 & 0 & 0.06 & 0 & 0 & -0.05 & 0 & 0 & 0 & 0 \\
0 & 0 & 0.45 & 0 & 0 & 0.05 & 0 & 0 & 0 & 0 & 0.1
\end{array} \right],
\end{equation*}
\end{footnotesize}
which is designed to be sparse in an expressive nonlinear feature space $\phi(x)$ to incorporate the SINDy framework \cite{brunton2016discovering}.
The true trajectory of the system is generated from
\begin{align*}
    x_0 = [3,3,3]^T. \quad x_{t+1} = \bar A \phi(x_t) + w_t, ~~t=0,\dots, T-1,
\end{align*}
where $T=2500$.  We consider two  disturbance scenarios for $w_t \in \mathbb{R}^3$: (a) a symmetric case where each component is independently uniform on $[-0.2, 0.2]$, and (b) a sparse case where $w_t$ equals the zero vector with probability $0.6$, and with probability $0.4$, its components are uniformly distributed on $[0.2-\min\{\|x_t\|_2, 1\},~ 0.2+\min\{\|x_t\|_2, 1\}]$, deliberately designed to depend on $x_t$. Under each of these disturbances, we run \eqref{ls}, \eqref{l1}, and \eqref{hubermin} with $\mu=0.1$ to obtain estimates $\hat A$. 

Figure \ref{ex2} validates the theoretical error bounds derived in Theorems \ref{noisescenario} and \ref{attackscenario} by plotting the Frobenius error norm $\|\bar A - \hat A\|_F$, against the trajectory length $T \in [40, 2500]$. Under persistent zero-mean noise (Figure \ref{fig:zerofrob}), the least-squares estimator converges at a rate of $\mathcal{O}(1/\sqrt{T})$, with the Huber estimator matching this rate up to a constant factor. Under sparse nonzero-mean attack (Figure \ref{fig:nonzerofrob}), the $\ell_1$-norm estimator perfectly recovers the system with zero error for $T \ge 100$. The Huber estimator yields a bounded constant error of $\mathcal{O}(\mu)$, significantly outperforming the least-squares approach. Ultimately, these results confirm that the Huber estimator serves as a robust bridge between standard mean- and median-based estimators.

\begin{figure}[t]
     \centering
     \begin{subfigure}[b]{0.235\textwidth}
         \centering
    \includegraphics[width=\linewidth]{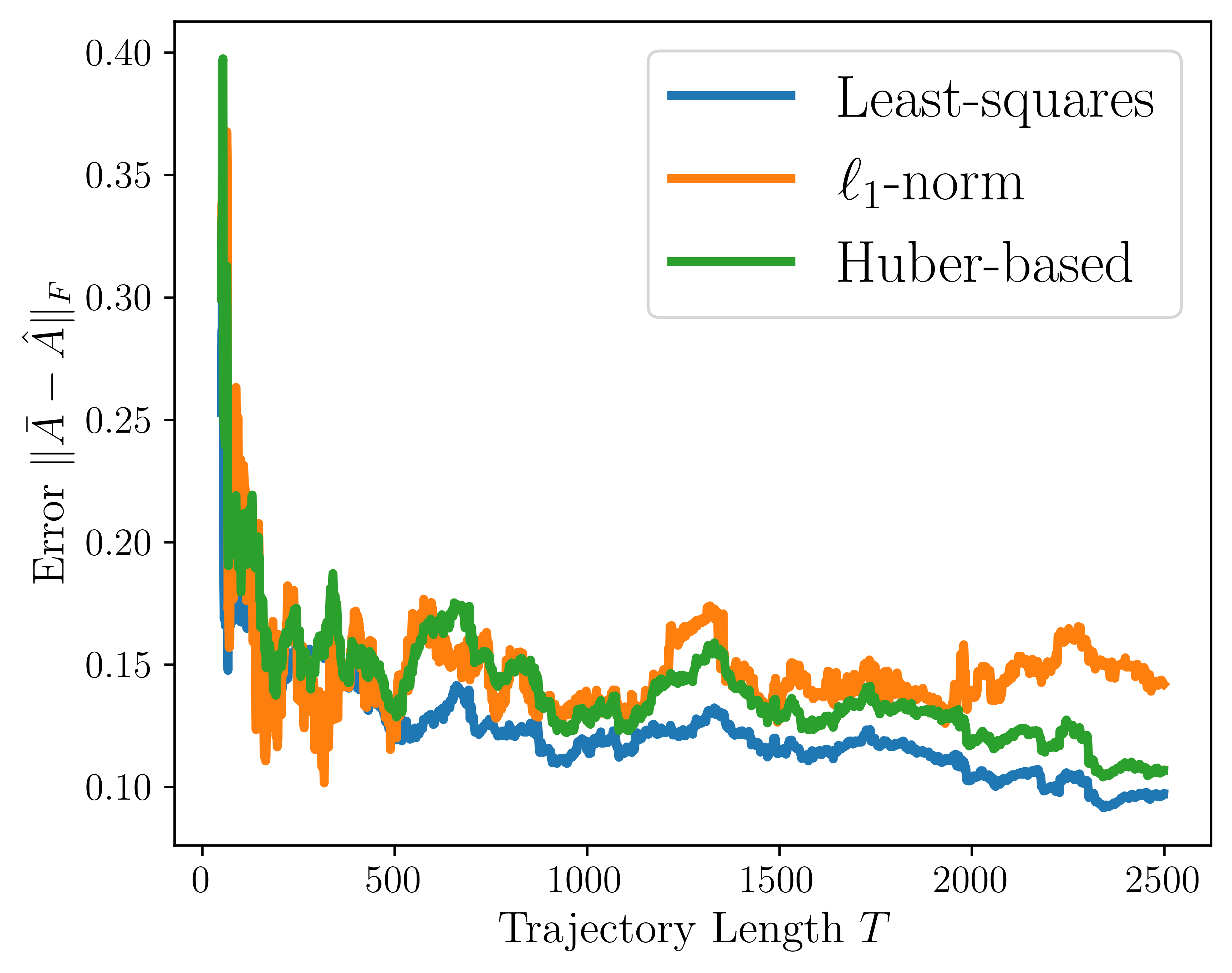}
         \caption{Persistent zero-mean noise}
         \label{fig:zerofrob}
     \end{subfigure}
     \begin{subfigure}[b]{0.2305\textwidth}
         \centering
    \includegraphics[width=\linewidth]{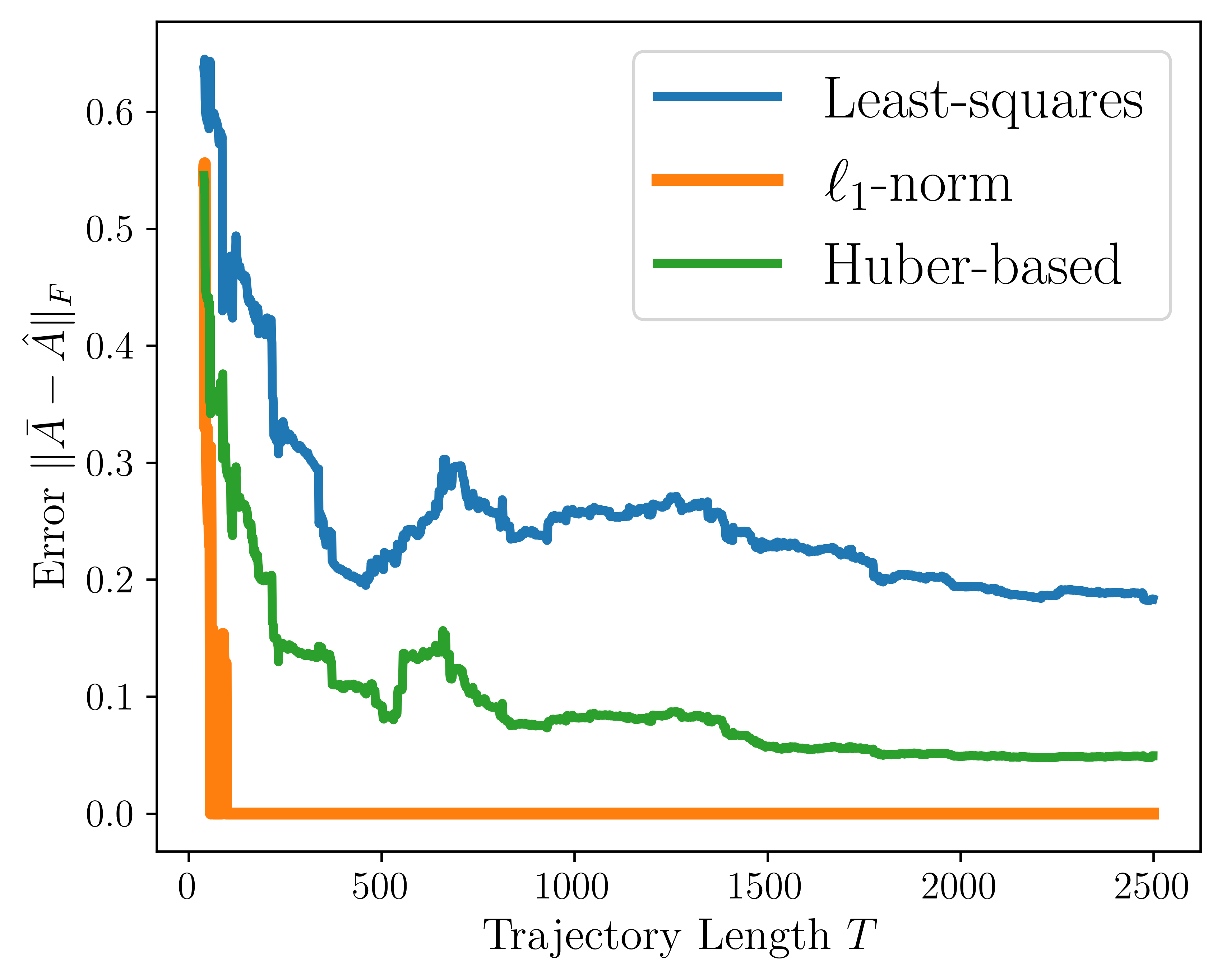}
         \caption{Sparse nonzero-mean attack}
         \label{fig:nonzerofrob}
     \end{subfigure}
        \caption{Estimation error over Trajectory Length}
        \label{ex2}
\end{figure}

We also investigate the stability of the true and reconstructed trajectories. 
Figures \ref{fig:zero}  and \ref{fig:nonzero} each display four trajectories: the true system path and those reconstructed by the three estimators at $T=2500$, with a ``Start'' marker and time-based coloring. Under persistent zero-mean noise in Figure \ref{fig:zero}, the true trajectory is stable; the least-squares and Huber estimators both successfully reconstruct stable paths that closely track the truth, whereas the $\ell_1$-norm estimate fails to stabilize. Conversely, under  sparse nonzero-mean attack in Figure \ref{fig:nonzero}, the $\ell_1$-norm estimator achieves a perfect reconstruction of the true stable trajectory, while the least-squares estimate diverges. The Huber estimator, however, continues to successfully produce a stable trajectory that approximates the truth. This demonstrates that the Huber-based estimator, with an appropriate value for $\mu$, reliably achieves accurate reconstruction under both extreme disturbance scenarios.

\begin{figure}
    \centering
    \includegraphics[width=0.98\linewidth]{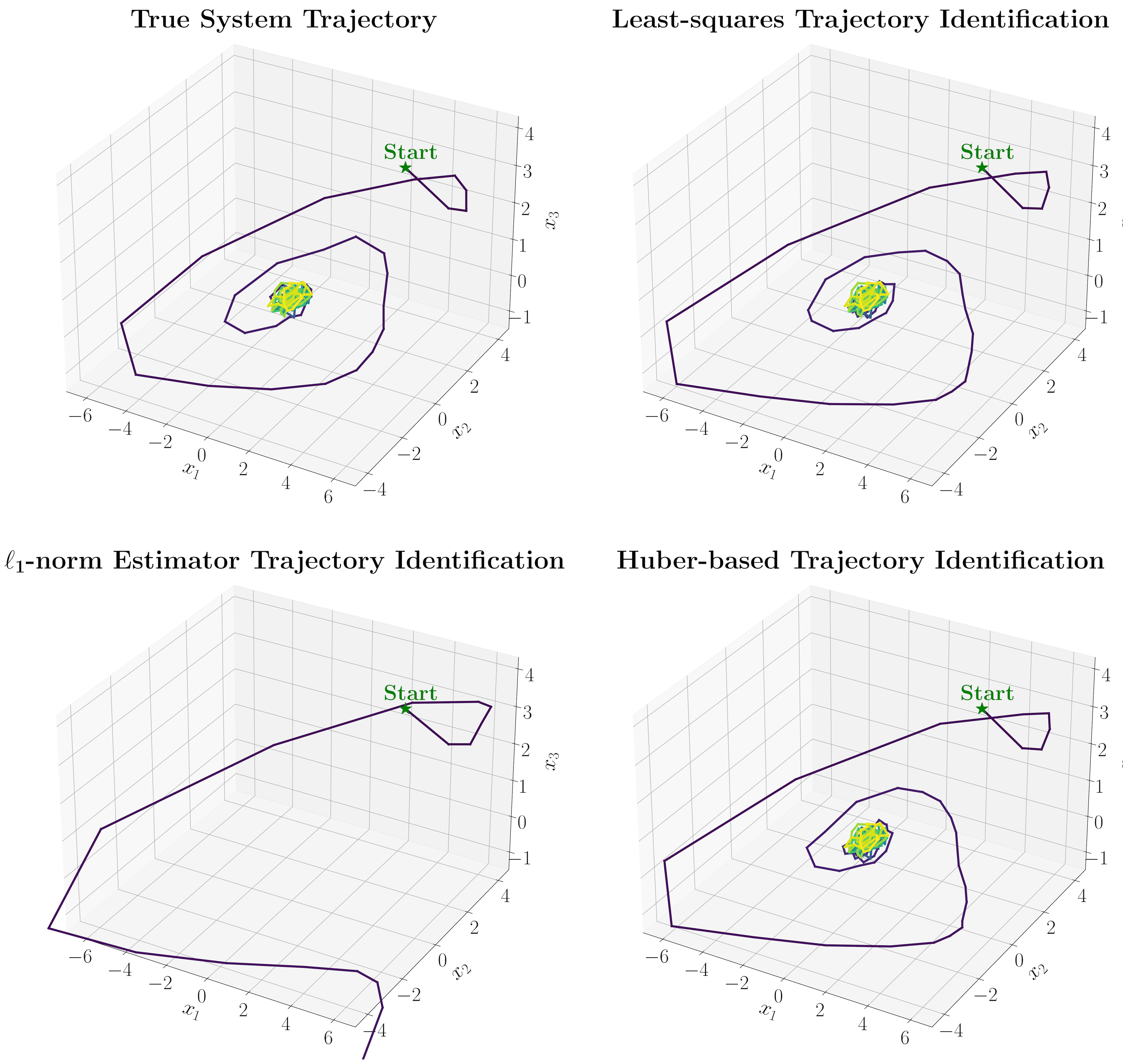}
    \caption{Stability of trajectories reconstructed based on different estimators: Persistent zero-mean independent noise}
    \label{fig:zero}
\end{figure}

\begin{figure}
    \centering
    \includegraphics[width=0.98\linewidth]{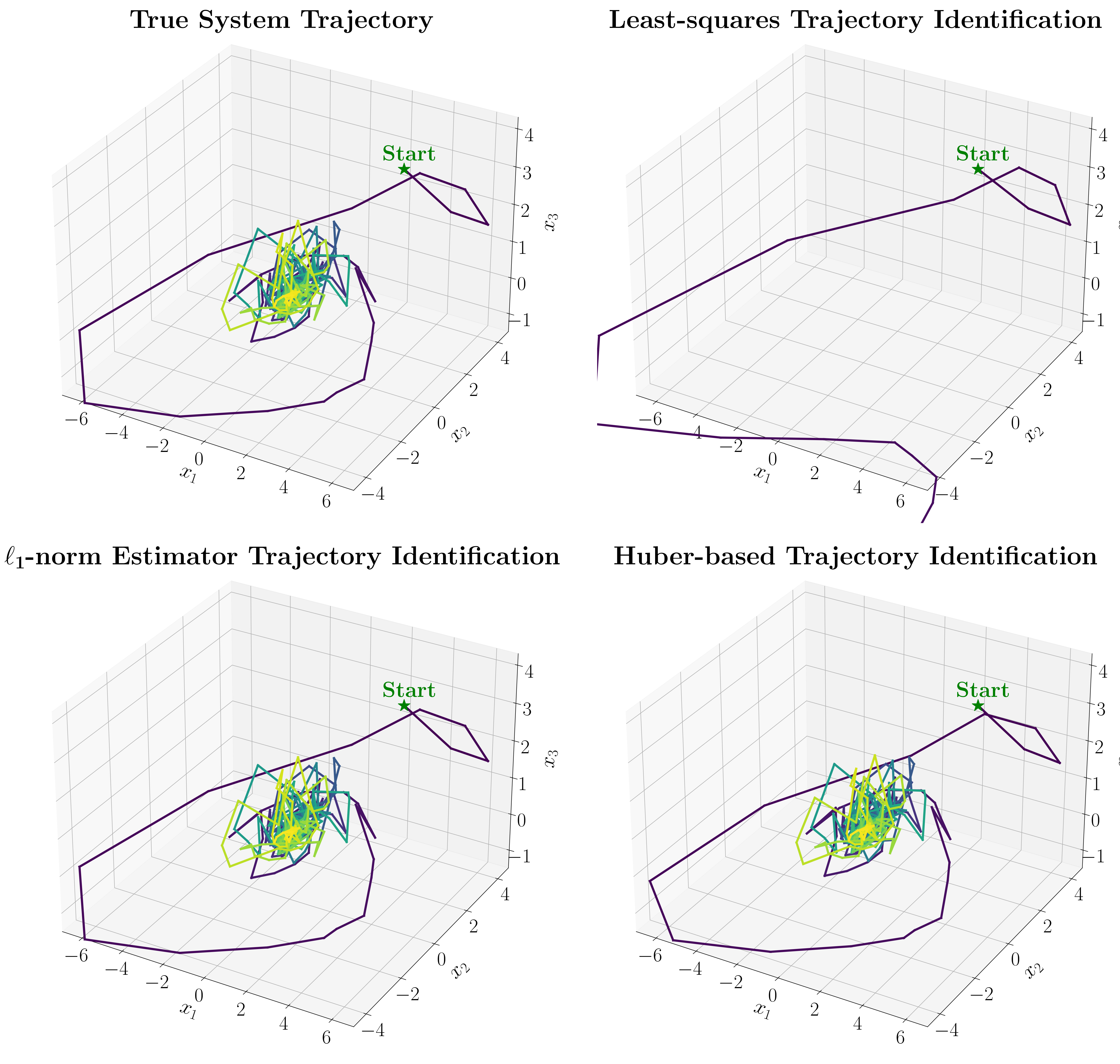}
    \caption{Stability of trajectories reconstructed based on different estimators: Sparse nonzero-mean adversarial attack}
    \label{fig:nonzero}
\end{figure}

\section{Proof of Theorem \ref{noisescenario}}\label{sec:proofthm1}

In this section, we prove Theorem \ref{noisescenario}, which shows the efficacy of the Huber estimator under Scenario \ref{as:white}. 
We begin by presenting a useful lemma on vector-valued martingales from \cite{juditsky2008large}. 

\begin{lemma}[\cite{juditsky2008large}]\label{vectorvalued}
Consider a filtration $ \mathcal{A}_{-1}\subset \mathcal{A}_0\subset \dots$ and a vector-valued martingale difference sequence $\{\xi_t\}_{t\ge 0}$ adapted to $\{\mathcal{A}_{t}\}$; \textit{i.e.}, $\mathbb{E}[\xi_t \,|\,\mathcal{A}_{t-1}] = 0$, $\xi_t$ is $\mathcal{A}_{t}$-measurable and square-integrable for all $t\ge 0$. Let $\kappa =1$ if $\{\xi_t\}\subset\mathbb{R}^m$ and $\kappa = 1+2\log m$ if $\{\xi_t\}\subset \mathbb{R}^{m\times m}$ are symmetric.
\begin{enumerate}[leftmargin=0.45cm]
    \item Let $\sigma_t\ge 0$ satisfy $\mathbb{E}[\exp (\|\xi_t\|_2^2 / \sigma_t^2)\,|\,\mathcal{A}_{t-1}] \le e$.  Then,
    \begin{align*}
        \mathbb{P}\biggr(\Bigr\|\sum_{t=0}^{T-1} \xi_t \Bigr\|_2 \ge \sqrt{2}(\sqrt{\kappa} + s) \sqrt{\sum_{t=0}^{T-1} \sigma_t^2}  \biggr) \le \exp\Bigr(-\frac{s^2}{3}\Bigr).
    \end{align*}
    \item  Let $\sigma_t\ge 0$ satisfy  $\mathbb{E}[\exp (\|\xi_t\|_2 / \sigma_t)\,|\,\mathcal{A}_{t-1}] \le e$. Then,
    \begin{align*}
      &  \mathbb{P}\biggr(\Bigr\|\sum_{t=0}^{T-1} \xi_t \Bigr\|_2 \ge \sqrt{2}(\sqrt{e\kappa}+s) \sqrt{\sum_{t=0}^{T-1} \sigma_t^2}  \biggr) \\&\hspace{30mm}\le 2\exp\Bigr(-\frac{\min\{s^2, 16\tau s\}}{64}\Bigr),
    \end{align*}
    where $\tau = \frac{\|(\sigma_0,\dots, \sigma_{T-1})\|_2}{\max_{t\in\{0,\dots, T-1\}} |\sigma_t|}$.
\end{enumerate}
\end{lemma}

The next lemma bounds the $\psi_2$-norm of $\|\phi(x_t)\|_2 $ for all $t\ge 0$.

\begin{lemma}\label{phixtnorm}
   Under Assumptions  \ref{as:stability} and \ref{as:subg}, we have $\bigr\| \| \phi(x_t) \|_2\bigr\|_{\psi_2}\le \frac{ \sqrt{n}L \sigma}{1-\rho L}$
   for every $t=0,\dots, T-1$.
\end{lemma}

\begin{proof}
  For $k=0, 1, \dots$, define the sequence $\{x_t^k\}$ as the state trajectory generated from 
    \begin{align*}
  x_0^k = x_0,\quad     x_{t+1}^k = \bar A \phi(x_{t}^k ) + w_{t}^k, \quad t=0,\dots, T-1,
    \end{align*}
    where $w_t^k$ is the truncated noise defined as $w_t$ for $t<k$ and equals zero otherwise.
 For notational simplicity, let $x_t^{-1} =0$ for all $t\ge 0$.  Noting that $x_t^t=x_t$, we can then establish that 
    \begin{align*}
          \|\phi(x_t)&\|_{2} = \biggr\| \sum_{j=0}^{t}\bigr( \phi(x_t^{j})-\phi(x_t^{j-1})\bigr)\biggr\|_{2}\\&\le \sum_{j=0}^{t} \|\phi(x_t^{j})-\phi(x_t^{j-1})\|_{2} \le  \sum_{j=0}^{t} L\|x_{t}^{j}-x_t^{j-1}\|_{2}\\&\le L(\rho L)^t \|x_0\|_{2} + L\sum_{j=1}^t  (\rho L)^{t-j} \|w_{j-1}\|_{2},\numberthis\label{Ltimesnorms}
    \end{align*}
    where  the last inequality comes from $\|x_{t}^{j}-x_t^{j-1}\|_{2}\le (\rho L)^{t-j} \|w_{j-1}\|_{2}$ since 
    \begin{align}\label{xtjdiff}
       \nonumber& x_{t}^{j}-x_t^{j-1} = (\bar A\phi)^{t-j} \left( \bar A\phi(\dots \bar A\phi(x_0) + w_0 \dots) + w_{j-1} \right)\\&\hspace{13mm} - (\bar A\phi)^{t-j} \left(\bar  A\phi(\dots \bar A\phi(x_0) + w_0 \dots) + 0 \right).
    \end{align}
    Similarly,  we have $x_t^0 = (\bar A\phi)^{t}(x_0)$, which implies that $\|x_t^0\|_{2}\le (\rho L)^t \|x_0\|_{2}$. 
   We
    substitute upper bounds on each $\psi_2$-norm  (see Remark \ref{subgnorm} that $\|x_0\|_2, \|w_{j-1}\|_2$ have $\psi_2$-norm of $\sqrt{n} \sigma$), and 
    use the geometric sum to conclude that \eqref{Ltimesnorms} is bounded by $\frac{L}{1-\rho L}\cdot  \sqrt{n}\sigma$. 
\end{proof}

The next lemma shows that either symmetric disturbances with nonlinear basis functions or generic zero-mean disturbances in linear systems provide tractable theoretical bounds.

\begin{lemma}\label{secondterm}
    Define 
 \begin{equation}\label{huberderivative}
       H_\mu'(z)=\begin{dcases*}
          z, & if $|z|\le \mu$\\
            \mu, & if $z> \mu$,\\
            -\mu, & if $z< -\mu$,
        \end{dcases*}
    \end{equation}
which is    the first derivative of the Huber function $H_\mu(z)$. 
  Consider Scenario \ref{as:white} and suppose that Assumptions \ref{as:stability}, \ref{as:subg},  and \ref{as:mildinnoise} hold. 
    Given $\delta\in(0,1)$, when 
    \begin{itemize}
       \item $w_t$ has a symmetric distribution for all $t\ge 0$, or
        \item $\phi(x) = \bar B x$ for some $\bar B\in\mathbb{R}^{m\times n}$, where $\|\bar B\|_2\le L$,
    \end{itemize}
    the bound
    \begin{align*}
      & \left\| \sum_{t=0}^{T-1} H_\mu'(w_t^{(i)}) \phi (x_t)\right\|_2= \mathcal{O}\biggr(\frac{\sqrt{T} \mu \sqrt{n}L\sigma}{1-\rho L}\log\Bigr(\frac{1}{\delta}\Bigr)\biggr)
    \end{align*}
holds with probability at least $1-\delta$.    
\end{lemma}

\begin{proof}
By the triangle inequality, we have
\begin{align*}
&\left\| \sum_{t=0}^{T-1} H_\mu'(w_t^{(i)}) \phi (x_t)\right\|_2 \le   \underbrace{\left\| \sum_{t=0}^{T-1} \mathbb{E}[H_\mu'(w_t^{(i)})] \phi (x_t)\right\|_2}_{(A)}\\&\hspace{18mm}+\underbrace{\left\| \sum_{t=0}^{T-1} \bigr(H_\mu'(w_t^{(i)}) - \mathbb{E}[H_\mu'(w_t^{(i)})] \bigr) \phi (x_t)\right\|_2.}_{(B)}
\end{align*}
We now separately analyze each term. 

\medskip
\textbf{Term (A)—\textit{Case 1}:}
In the case where $w_t$ is  symmetric, we have $\mathbb{E}[H_\mu'(w_t^{(i)})]=0$ for all $t\ge 0$, regardless of the value of $\mu$, which implies that $\text{Term~}(A)=0$. 

\textbf{Term (A)—\textit{Case 2}:}
We now consider the case where $\phi(x)$ is linear in $x$; \textit{i.e.}, $\phi(x) = \bar Bx$. For notational simplicity, let $w_{-1} = x_0$. Recalling the definition of $\{x_t^k\}$ in the proof of Lemma \ref{phixtnorm}, we rewrite the term as 
\begin{align*}
    \text{Term~}(&A)= \Biggr\|\sum_{t=0}^{T-1} \mathbb{E}[H_\mu'(w_t^{(i)})]  \sum_{j=0}^{t}\bigr( \phi(x_t^{j})-\phi(x_t^{j-1}) \bigr)\Biggr\|_2 \\&= \Biggr\|\sum_{j=0}^{T-1} \sum_{t=j}^{T-1} \mathbb{E}[H_\mu'(w_t^{(i)})]\cdot \bar B (\bar A\bar B)^{t-j} w_{j-1}\Biggr\|_2,\numberthis\label{termcrewrite}
\end{align*}
by considering \eqref{xtjdiff} in the linear case and interchanging the order of summation.  
   Define a filtration with inverse order
\begin{align*}
    \mathcal{G}_j = \bm{\sigma}\{w_{T-1}, w_{T-2}, \dots, w_j\}. 
\end{align*}
The term $\sum_{t=j}^{T-1} \mathbb{E}[H_\mu'(w_t^{(i)})]\cdot \bar B (\bar A\bar B)^{t-j} w_{j-1}$ is mean zero given $\mathcal{G}_j$ (since $\mathbb{E}[w_{j-1}\,|\,\mathcal{G}_j]=\mathbb{E}[w_{j-1}]=0$), and $\mathcal{G}_{j-1}$-measurable. Moreover, 
\begin{align*}
  &\Biggr \|  \biggr\|  \sum_{t=j}^{T-1} \mathbb{E}[H_\mu'(w_t^{(i)})]\cdot \bar B (\bar A\bar B)^{t-j} w_{j-1}\biggr\|_2   ~\biggr|~\mathcal{G}_j \Biggr\|_{\psi_2} \\&\le   \sum_{t=j}^{T-1}\mu  \|\bar B (\bar A\bar B)^{t-j}\|_2 \cdot \bigr\|\|w_{j-1}\|_2\bigr\| _{\psi_2}\le \frac{\mu L\cdot \sqrt{n} \sigma }{1-\rho L}
\end{align*}
holds since $|\mathbb{E}[H_\mu'(w_t^{(i)})]| \le \mu$, and apply geometric sum similar to the derivation of Lemma \ref{phixtnorm}.  
By applying the first property of Lemma \ref{vectorvalued}, there exists $c_1>0$ such that
\begin{align*}
   & \mathbb{P}\biggr( \biggr\|\sum_{j=0}^{T-1} \sum_{t=j}^{T-1} \mathbb{E}[H_\mu'(w_t^{(i)})]\cdot \bar B (\bar A\bar B)^{t-j} w_{j-1}\biggr\|_2\\&\hspace{25mm}\ge c_1 (1+s)\sqrt{T} \frac{\mu L \sqrt{n} \sigma}{1-\rho L}   \biggr) \le \exp\Bigr(-\frac{s^2}{3}\Bigr)
\end{align*}
This implies that 
\begin{align}\label{termA}
    \text{Term~}(A)= \mathcal{O}\biggr(  \frac{\sqrt{T}\mu \sqrt{n} L\sigma}{1-\rho L}\sqrt{\log\Bigr(\frac{1}{\delta}\Bigr)}\biggr)
\end{align}
holds with probability at least $1-\frac{\delta}{3}$. 

\medskip
\textbf{Term (B): } 
Since $\phi(x_t)$ is $\mathcal{F}_t$-measurable, the terms 
$(H_\mu'(w_t^{(i)})-\mathbb{E}[H_\mu'(w_t^{(i)})]) \phi(x_t)$ are $\mathcal{F}_{t+1}$-measurable with zero conditional mean given $\mathcal{F}_t$, thus forming a martingale difference sequence. 
We also have
\begin{align*}
  &\bigr\|  \|(H_\mu'(w_t^{(i)})-\mathbb{E}[H_\mu'(w_t^{(i)})]) \phi(x_t)\|_2~\bigr|~\mathcal{F}_t\bigr\|_{\psi_2} \\&\hspace{48mm}=\mathcal{O} \bigr(\mu \|\phi(x_t)\|_2\bigr),\numberthis\label{2muphi}
\end{align*}
 since $|H_\mu'(w_t^{(i)})-\mathbb{E}[H_\mu'(w_t^{(i)})]|\le 2\mu$. 

Note that $V_t:= \sum_{j=0}^{t-1}  \|\phi(x_j)\|_2^2$ is a sub-exponential variable with $\psi_1$-norm\footnote{For a sub-Gaussian variable $x$ with $\psi_2$-norm $\sigma$, $x^2$ is a sub-exponential variable with $\psi_1$-norm $\|x^2\|_{\psi_1} = \|x\|_{\psi_2}^2=\sigma^2$. The notion of sub-exponential variables are introduced in Section 2.8, \cite{vershynin2025high}.} $t\bigr(\frac{  \sqrt{n}L\sigma}{1-\rho L}\bigr)^2$, since we can apply Lemma \ref{phixtnorm} to $\|V_t\|_{\psi_1} \le \sum_{j=0}^{t-1}  \| \|\phi(x_j)\|_2^2 \|_{\psi_1}$. Thus, there exists a constant $c_2>0$ such that $V_t <  V_\text{max}$ holds for all $t\le T$ with probability at least $1-\frac{\delta}{3}$, where $V_{\text{max}}:= c_2 T\bigr(\frac{  \sqrt{n}L\sigma}{1-\rho L}\bigr)^2 \log\bigr(\frac{1}{\delta}\bigr)$. Then, define the stopping time:
\begin{align}\label{stoppingtime}
    t_{\text{stop}} = \inf\{t\ge 0: V_t\ge  V_{\text{max}}\}.
\end{align}

Now, we will bound the sum of the stopped martingale difference sequence by applying the first property of Lemma \ref{vectorvalued}. Considering \eqref{2muphi} and the definition of $V_{\text{max}}$, there exists $c_3 >0$ such that
\begin{align}\label{termbinter}
  \nonumber & \mathbb{P}\Biggr(\biggr\| \sum_{t=0}^{T-1} \bigr(H_\mu'(w_t^{(i)}) - \mathbb{E}[H_\mu'(w_t^{(i)})] \bigr) \phi (x_t)\mathbb{I}\{t\le t_{\text{stop}}\}\biggr\|_2\\&\hspace{18mm}\ge c_3(1+s)  \sqrt{  \mu^2 V_{\text{max}}}\Biggr)\le \exp\Bigr(-\frac{s^2}{3}\Bigr),
\end{align}
where  $\mathbb{I}\{\cdot \}$ denotes the indicator function.
 Restricting to the event  \{$V_T<V_\text{max}$\}, we have $t_{\text{stop}}\ge T$, which implies that $\mathbb{I}\{t\le t_{\text{stop}}\}=1$ for all $t\in\{0,\dots, T-1\}$, from which
 \eqref{termbinter} implies that
\begin{align}\label{termB}
 \nonumber   \text{Term~} (B) &= \mathcal{O}\biggr(  \mu \sqrt{V_\text{max}} \sqrt{ \log\Bigr(\frac{1}{\delta}\Bigr)}\biggr) \\&= \mathcal{O}\Bigr(  \frac{\sqrt{T}\mu  \sqrt{n}L\sigma}{1-\rho L}\log\Bigr(\frac{1}{\delta}\Bigr)   \Big)
\end{align}
holds with probability at least $1-\frac{\delta}{3}$.

Note that an upper bound on Term $(B)$ given in \eqref{termB} dominates both cases for Term $(A)$ in order.  We complete the proof by applying the union bound over 
$\{V_T< V_\text{max}\}$, \eqref{termB},  and include \eqref{termA} in \textbf{(A)-\textit{Case 2}}. 
\end{proof}

The next lemma establishes that sufficient expected excitation implies sufficient empirical excitation with high probability. 
\begin{lemma}\label{xxlowerbound}
    Suppose that Assumptions \ref{as:stability}, \ref{as:subg}, and  \ref{as:excitation} hold.  Consider a subset of $\{1,\dots, T\}$  to be $\mathcal{T}$ and its cardinality as $|\mathcal{T}|$. Given $\delta\in(0,1)$, when 
\begin{align}\label{finaltimebound}
   |\mathcal{T}| = \Omega\left( \frac{ (\sqrt{n}L \sigma)^4}{\lambda^4 (1 - \rho L)^2}  \log\Bigr(\frac{m}{\delta}\Bigr)\log\Bigr(\frac{1}{\delta}\Bigr)\right),
\end{align}
we have $\sum_{t\in \mathcal{T}} \phi(x_t)\phi(x_t)^T \succeq \frac{\lambda^2 I}{2}|\mathcal{T}|$ with probability at least $1-\delta$.
\end{lemma}
\begin{proof}
By Assumption \ref{as:excitation}, we have $\sum_{t\in\mathcal{T}}\mathbb{E}[\phi(x_t) \phi(x_t)^T\,|\, \mathcal{F}_{t-1}]\succeq |\mathcal{T}|\lambda^2 I$. To arrive at the conclusion, it suffices to prove that 
\begin{align*}
   \biggr\| \sum_{t\in \mathcal{T}} \phi(x_t) \phi(x_t)^T -\mathbb{E}[\phi(x_t) \phi(x_t)^T\,|\, \mathcal{F}_{t-1}] \biggr\|_2 \le \frac{|\mathcal{T}|\lambda^2}{2}
\end{align*}
holds with probability at least $1-\delta$.

Now, define $\tilde x_t = x_t - w_{t-1} = \bar A \phi(x_{t-1})$ and consider the separation
\begin{align*}
   & \phi(x_t) = \underbrace{\phi(\tilde x_t) +\mathbb{E}[\phi(x_t) - \phi(\tilde x_t) \mid\mathcal{F}_{t-1}]}_{A_t} \\&\hspace{20mm}+ \underbrace{\phi(x_t) - \phi(\tilde x_t) -\mathbb{E}[\phi(x_t) - \phi(\tilde x_t) \mid\mathcal{F}_{t-1}] }_{B_t}.
\end{align*}
   Noting that the term $A_t$ is $\mathcal{F}_{t-1}$-measurable, and the term $B_t$ has mean zero given $\mathcal{F}_{t-1}$, we have 
   \begin{align*}
     \phi(x_t) \phi(x_t)^T&-  \mathbb{E}[\phi(x_t)\phi(x_t)^T\,|\,\mathcal{F}_{t-1}]
    \\&=\underbrace{A_t B_t^T + B_tA_t^T}_{C_t} +\underbrace{ B_tB_t^T   -\mathbb{E}[B_t B_t^T\,|\,\mathcal{F}_{t-1}]}_{D_t}.
   \end{align*}
  Thus, we need to bound $\|\sum_{t\in \mathcal{T}} C_t + D_t\|_2$. To this end, we will  separately bound $\|\sum_{t\in \mathcal{T}}C_t\|_2$ and $\|\sum_{t\in \mathcal{T}}D_t\|_2$. Note that both $C_t$ and $D_t$ are $\mathcal{F}_t$-measurable and mean-zero given $\mathcal{F}_{t-1}$.

 First, noting that $\|C_t\|_2 \le 2\|A_t\|_2 \|B_t\|_2$, we have 
  \begin{align*}
      \bigr\|  \|C_t\|_2 \,&\bigr|\,\mathcal{F}_{t-1} \bigr\|_{\psi_2}\le 2\|A_t\|_2 \cdot \bigr\| \|B_t\|_2 \,\bigr|\,\mathcal{F}_{t-1} \bigr\|_{\psi_2}\\&\le 2\|A_t\|_2 \cdot 2L\bigr\|\|w_{t-1}\|_2\bigr\|_{\psi_2}\le 4\|A_t\|_2 \cdot L\sqrt{n}\sigma,
  \end{align*}
  since \begin{align}\label{xttildext}
      \|\phi(x_t)-\phi(\tilde x_t)\|_2 \le L\|x_t-\tilde x_t\|_2 =  L\|w_{t-1}\|_2.
  \end{align} 
  Note that $W_t:= \sum_{j\in\mathcal{T}\cap \{1,\dots, t\}} \|A_j\|_2^2$ is a sub-exponential variable with $\psi_1$-norm of at most $|\mathcal{T}|\bigr(\frac{2\sqrt{n}L\sigma}{1-\rho L}\bigr)^2$ (since $\|\phi(\tilde x_t)\|_2$ and $\mathbb{E}[\|\phi(x_t) - \phi(\tilde x_t)\|_2\,|\,\mathcal{F}_{t-1}]$ both have $\psi_2$-norm of at most $\frac{\sqrt{n}L\sigma}{1-\rho L}$). Thus, there exists a constant $c_5>0$ such that $W_t <W_{\text{max}}$ holds for all $t\le T$ with probability at least $1-\frac{\delta}{3}$, where $W_\text{max} := c_5 |\mathcal{T}|\bigr(\frac{\sqrt{n}L\sigma}{1-\rho L}\bigr)^2\log\big(\frac{1}{\delta}\big)$.

Using the same logic of the stopping time argument for the martingale difference sequence as in \eqref{stoppingtime}-\eqref{termB}, we can again  apply the first property of Lemma \ref{vectorvalued}. Then, restricting to the event $\{W_T<W_\text{max}\}$, there exists $c_4>0$ such that 
  \begin{align}\label{ct}
     \nonumber \mathbb{P}\biggr(\Bigr\|\sum_{t\in\mathcal{T}} C_t\Bigr\|_2 \ge c_4(\sqrt{1+\log m}+&s)\sqrt{ (\sqrt{n}L\sigma)^2 W_\text{max}}\biggr)\\&\le \exp\Bigr(-\frac{s^2}{3}\Bigr),
  \end{align}
which implies that 
  \begin{align}\label{sumct}
 \nonumber  &   \Bigr\| \sum_{t\in\mathcal{T}} C_t\Bigr\|_2 \\&=\mathcal{O}\biggr( \Bigr( \sqrt{\log m}+ \sqrt{\log\Bigr( \frac{1}{\delta}\Bigr)} \Bigr) \sqrt{|\mathcal{T}|}\frac{ (\sqrt{n}L \sigma)^2}{1-\rho L}\sqrt{\log\Bigr( \frac{1}{\delta}\Bigr)}\biggr)
  \end{align}
holds   with probability at least $1-\frac{\delta}{3}$.

  Second, noting that $\|B_t B_t^T\|_2 = \|B_t\|_2^2$, we have
  \begin{align*}
    &  \bigr\| \|D_t\|_2\,|\,\mathcal{F}_{t-1}  \bigr\|_{\psi_1}\le 2 \bigr\| \|B_t\|_2^2\,|\,\mathcal{F}_{t-1}  \bigr\|_{\psi_1}\\&\le 2 \bigr\| (2L\|w_{t-1}\|_2)^2  \bigr\|_{\psi_1}= 2 \bigr\| 2L\|w_{t-1}\|_2  \bigr\|_{\psi_2}^2\le 8(\sqrt{n}L\sigma)^2,
  \end{align*}
  where the second inequality follows from \eqref{xttildext}. From the second property of Lemma \ref{vectorvalued}, there exists $c_6>0$ such that
  \begin{align*}
    &  \mathbb{P}\biggr(  \Bigr\| \sum_{t\in\mathcal{T}} D_t  \Bigr\|_2 \ge c_6 (\sqrt{1+\log m} +s ) \sqrt{|\mathcal{T}|} (\sqrt{n}L\sigma)^2 \biggr)\\&\hspace{35mm}\le 2\exp\Bigr( -\frac{\min\{ s^2, 16\sqrt{|\mathcal{T}|}s\} }{64} \Bigr).
  \end{align*}
  This implies that 
  \begin{align}\label{sumdt}
      \Bigr\| \sum_{t\in\mathcal{T}} D_t  \Bigr\|_2 = \mathcal{O}\biggr( \Bigr( \sqrt{\log m}+  \log\Bigr(\frac{1}{\delta}\Bigr)  \Bigr) \sqrt{|\mathcal{T}|} (\sqrt{n}L \sigma)^2  \biggr)
  \end{align}
  with probability at least $1-\frac{\delta}{3}$. 
Since  \eqref{sumct} dominates \eqref{sumdt},  $\bigr\|  \sum_{t\in\mathcal{T}} C_t + D_t \bigr\|_2$ is bounded by  \eqref{sumct} with probability at least $1-\delta$, 
by constructing the union bound over $\{W_T<W_\text{max}\}$, \eqref{sumct}, and \eqref{sumdt}. 
For \eqref{sumct} to be bounded by $\frac{|\mathcal{T}|\lambda^2}{2}$, it suffices for $|\mathcal{T}|$ to satisfy \eqref{finaltimebound}.  This completes the proof.
\end{proof}

\medskip
Now, we are finally ready to prove our main theorem that validates that the Huber estimator obtains $\mathcal{O}(1/\sqrt{T})$ error under persistent zero-mean independent noise process. 

\textit{\textbf{Proof of Theorem \ref{noisescenario}}}:
   Since $H_\mu$ is convex, the first-order conditions provide necessary and sufficient conditions for optimality of \eqref{hubermin}.

 Then, considering that $x_{t+1}^{(i)} = \bar a_i^T \phi(x_t) +w_t^{(i)}$,  
and denoting $\bar \epsilon_i = \bar a_i - \hat a_i$ for each $i$, we have 
 \begin{align}\label{firstorder}
     \sum_{t=0}^{T-1}  H_{\mu}' ( \bar \epsilon_i^T \phi(x_t) + w_t^{(i)} )\cdot  \phi(x_t) = 0 , ~~ \forall i\in\{1,\dots,n\},
 \end{align}
 where $H_\mu'(z)$ is defined in \eqref{huberderivative}. 
 Let 
 \begin{align*}
     F^{(i)} (\epsilon) :=  \sum_{t=0}^{T-1}  H_{\mu}' (  \epsilon^T \phi(x_t) + w_t^{(i)} )\cdot  \phi(x_t).
 \end{align*}
 We define the following time index set  
\begin{align*}
   & S_i(\epsilon) = \{t\in\{0,\dots, T-1\}: |\epsilon^T \phi(x_t) + w_t^{(i)}|\le \mu\}. 
\end{align*}
Then, the Jacobian of $F^{(i)} (\epsilon)$ is 
defined as 
    $J^{(i)}(\epsilon) = \sum_{t\in S_i(\epsilon) }\phi(x_t)\phi(x_t)^T,$
since the second derivative of $H_\mu(z)$ is $1$ if $|z|\le \mu$ and $0$ otherwise. By the fundamental theorem of calculus, we have
\begin{align}\label{fund}
    F^{(i)} (\epsilon)-F^{(i)} (0) = \int_{0}^1   J^{(i)}(s \epsilon) \epsilon ~ds.
\end{align}

  Let $\Gamma_i = \{t\in\{0,\dots, T-1\}: |w_t^{(i)}| \le \frac{\mu}{2}\}$. 
   We now consider the events
   \begin{align*}
     &  \mathcal{E}_1 = \Bigr\{|\Gamma_i|\ge \frac{qT}{2}\Bigr\}, \\& \mathcal{E}_2 = \bigr\{  |\epsilon^T \phi(x_t)| \le \frac{\mu}{2}, ~~\forall t=0,\dots, T-1\bigr\}
   \end{align*}
Under these two events, we know that there exists a time index set $\tilde{\Gamma}_i(\epsilon)$ with a cardinality of at least $\frac{qT}{2}$, such that for all $t \in \tilde{\Gamma}_i(\epsilon)$, $|w_t^{(i)}| \le \frac{\mu}{2}$ and $|\epsilon^T \phi(x_t)| \le \frac{\mu}{2}$. More importantly, $\tilde{\Gamma}_i(\epsilon)$ is a subset of $S_i(s\epsilon)$ for any $0 \le s \le 1$.
 Then, we consider the third event
\begin{align*}
    \mathcal{E}_3 = \Bigr\{  \sum_{t\in \tilde \Gamma_i(\epsilon)} \phi(x_t) \phi(x_t)^T \succeq \frac{\lambda^2 I}{2} \cdot \frac{qT}{2}  \Bigr\}.
\end{align*}
Under this event, we know that $J^{(i)}(s\epsilon)\succeq \frac{\lambda^2 qT }{4}I$ for any $0\le s\le 1$. 
By multiplying $\epsilon^T$ to both sides of \eqref{fund}, we obtain 
\begin{align*}
    \epsilon^T F^{(i)}(\epsilon) -\epsilon^T& F^{(i)}(0) = \int_0^1 \epsilon^T J^{(i)}(s\epsilon) \epsilon~ ds\\&\ge \int_0^1 \frac{\lambda^2 qT }{4}\|\epsilon\|_2^2 ~ds =\frac{\lambda^2 qT }{4}\|\epsilon\|_2^2.\numberthis\label{epseps}
\end{align*}
Now, we will measure the probability that the events $\mathcal{E}_1, \mathcal{E}_2, \mathcal{E}_3$ hold simultaneously. 
For $\mathcal{E}_1$, one can apply the Chernoff bound under Assumption \ref{as:mildinnoise} to obtain $\mathbb{P}(|\Gamma_i| \ge \frac{qT}{2})\ge 1-\exp(-\frac{qT}{8})$, which implies that 
 $T=\Omega\bigr( \frac{1}{q} \log(\frac{1}{\delta})\bigr)$ ensures that $\mathbb{P}(\mathcal{E}_1)\ge1-\frac{\delta}{6}$.

For $\mathcal{E}_2$, we construct a union bound over all $t$ to obtain $\max_{t\in\{0,\dots,T-1\}} \|\phi(x_t)\|_2  = \mathcal{O} \bigr(\frac{\sqrt{n} L\sigma}{1-\rho L}\sqrt{\log(\frac{T}{\delta})}\bigr)$ as
\begin{align*}
       & \mathbb{P}\bigr(\max_{t\in\{0,\dots, T-1\}}\|\phi(x_t)\|_2\ge s\bigr) \\&= \mathbb{P}\bigr(\{\|\phi(x_0)\|_2\ge s\}\cup \dots \cup \{\|\phi(x_{T-1})\|_2\ge s\}\bigr)\\&\le T \cdot 2\exp\Bigr(-\Omega\Bigr(\frac{s^2}{(\sqrt{n}L\sigma  / (1-\rho L))^2}\Bigr)\Bigr),
    \end{align*}
    for all $s\ge 0$. Considering that $|\epsilon^T \phi(x_t)|\le \|\epsilon\|_2 \|\phi(x_t)\|_2$, 
there exists a constant $c_7>0$ such that 
\begin{align*}\label{epsR}
   & \|\epsilon\|_2 \le \frac{c_7\mu (1-\rho L)}{\sqrt{n}L \sigma\sqrt{\log(T/\delta)} }:=R ~~\Longrightarrow  ~~\mathbb{P}\left(\mathcal{E}_2 \right)\ge 1-\frac{\delta}{6}.
\end{align*}

For $\mathcal{E}_3$, since $|\tilde \Gamma_i(\epsilon)|\ge\frac{qT}{2}$ under $\mathcal{E}_1\cap \mathcal{E}_2$, 
when $\frac{qT}{2}$ satisfies the time complexity \eqref{finaltimebound}, $\mathcal{E}_3$ holds with probability at least $1-\frac{\delta}{6}$. Taking the union bound, 
when $\|\epsilon\|_2\le R$ and $\frac{qT}{2}$ satisfies the time complexity \eqref{finaltimebound} (which already subsumes $T=\Omega(\frac{1}{q}\log(\frac{1}{\delta}))$),
we have $\mathbb{P}(\mathcal{E}_1\cap \mathcal{E}_2 \cap \mathcal{E}_3)\ge 1-\frac{\delta}{2}$; as a result, \eqref{epseps} holds with probability at least $1-\frac{\delta}{2}$. 

Meanwhile, $\|F^{(i)}(0)\|_2$ is bounded by $\mathcal{O}(\sqrt{T})$ with probability at least $1-\frac{\delta}{2}$, which follows from Lemma 
  \ref{secondterm}. In this case, large enough $T$ ensures that $\frac{\lambda^2 qT}{4}R -\|F^{(i)}(0)\|_2 > 0$. In particular, using the fact that $T \ge 2A \log(A/\delta)$ implies  $\frac{T}{\log(T/\delta)} > A$, it suffices to have
 \begin{align*}
    & T = {\Omega} \biggr( \frac{(\sqrt{n} L\sigma)^4 }{ q^2\lambda^4 (1-\rho L)^4} \log^2\left(\frac{m}{\delta}\right)\log\left( \frac{Ln\sigma }{q\lambda (1-\rho L) \delta} \right) \biggr)\numberthis\label{timebound4}
 \end{align*}
 for $\frac{\lambda^2 qT}{4}R -\|F^{(i)}(0)\|_2 > 0$ to hold  with probability at least $1-\frac{\delta}{2}$. In such a case, as a by-product, we have $R=\mathcal{O}\bigr(\frac{ \mu  \sqrt{n}L\sigma}{\sqrt{T}q\lambda^2 (1-\rho L)}\log\bigr(\frac{1}{\delta}\bigr)\bigr).$
Note that \eqref{timebound4} implies that 
$\frac{qT}{2}$ satisfies the time complexity \eqref{finaltimebound}. Thus, we can use the union bound to establish that
 $\|\epsilon\|_2 =R$ in \eqref{epseps} implies that 
\begin{equation}\label{eps0}
\begin{split}
\epsilon^T F^{(i)}(\epsilon)&\ge  \frac{\lambda^2 qT}{4}\|\epsilon\|_2^2 - \|\epsilon\|_2 \|F^{(i)}(0)\|_2 \\&= R\Bigr(\frac{\lambda^2 qT}{4}R -\|F^{(i)}(0)\|_2 \Bigr)>0.
\end{split}
\end{equation}
  with probability at least $1-\delta$  under \eqref{timebound4}.

Note that $\bar \epsilon_i$ defined in \eqref{firstorder} cannot satisfy \eqref{eps0} since $F^{(i)}(\bar \epsilon_i)=0$. Meanwhile, by the continuity of $F^{(i)}(\epsilon)$,   $\|\epsilon\|_2 = R \Rightarrow  \epsilon^T F^{(i)}( \epsilon)>0$ implies that there exists $\|\epsilon\|_2<R$ such that $F^{(i)}( \epsilon)=0$
(see Theorem 6.3.4, \cite{Ortega2000}). Noting that a set of optimal points to convex optimization problem \eqref{hubermin} is indeed convex, every solution to \eqref{hubermin} should satisfy that $\|\epsilon\|_2 <R$. Under \eqref{timebound4}, we have 
\begin{align}\label{epsbound1}
    \|\epsilon\|_2<R= \mathcal{O}\biggr(\frac{ \mu  \sqrt{n}L\sigma}{\sqrt{T}q\lambda^2 (1-\rho L)}\log\Bigr(\frac{1}{\delta}\Bigr)\biggr)
\end{align}
with probability at least $1-\delta$. 

To ensure that the aforementioned argument holds for all $i\in\{1,\dots,n\}$, we substitute $\frac{\delta}{n}$ for $\delta$ in \eqref{timebound4} and \eqref{epsbound1} to  complete the proof.
\hfill $\blacksquare$

\section{Proof of Theorem \ref{attackscenario}}\label{sec:proofthm2}
In this section, we prove Theorem \ref{attackscenario}, which establishes that the estimation error of the Huber estimator is strictly bounded by a constant under Scenario \ref{as:prob}. 
The following lemma shows that  the $\ell_1$-norm estimator perfectly recovers the system with high probability.
\begin{lemma}\label{lem:excitation}
 Consider Scenario \ref{as:prob} and suppose that  Assumptions \ref{as:stability}, \ref{as:subg},   and \ref{as:excitation_attack} hold. 
    Let $f^{(i)}(a_i):= \sum_{t=0}^{T-1} |x_{t+1}^{(i)} -  a_i^T \phi(x_t)| $ for every $i$. 
    Given $\delta\in(0,1)$, when $T$ satisfies \eqref{timeboundattack}, 
there exists a constant $c>0$ such that
\begin{align*}
    f^{(i)}(a_i) - f^{(i)}(\bar a_i) \ge c T\cdot &\frac{p(1-2p)\lambda^5 }{n^2 L^4\sigma^4 } \|a_i - \bar a_i\|_2, \\&\forall a_i \in \mathbb{R}^m, \quad \forall i\in\{1,\dots, n\}
\end{align*}
with probability at least $1-\delta$.
As a by-product, 
$\bar A$ is the unique global solution to \eqref{l1} with probability at least $1-\delta$.
\end{lemma}

\begin{proof}
  Letting $u:= \frac{\bar a_i - a_i}{\|\bar a_i - a_i\|_2}\in\mathbb{R}^m$,  we have 
   \begin{align*}
      & f^{(i)}(a_i) - f^{(i)}(\bar a_i) = \sum_{t=0}^{T-1} |(\bar a_i-a_i)^T \phi(x_t) + w_t^{(i)}| -  |w_t^{(i)}|\\&\ge\sum_{\substack{t=0,\\ w_t^{(i)}=0}}^{T-1} |(\bar a_i-a_i)^T \phi(x_t) | + \sum_{\substack{t=0,\\ w_t^{(i)}\ne0}}^{T-1} (\bar a_i-a_i)^T \phi(x_t) \mathrm{sgn}(w_t^{(i)})\\&= \|\bar a_i-a_i\|_2 \biggr[\sum_{\substack{t=0,\\ w_t^{(i)}=0}}^{T-1} |u^T \phi(x_t) | + \sum_{\substack{t=0,\\ w_t^{(i)}\ne0}}^{T-1} u^T \phi(x_t)  \mathrm{sgn}(w_t^{(i)})\biggr],
   \end{align*}
where the first inequality follows from the gradient inequality for a convex function $|\cdot |$   and its subgradient being a sign function. To universally lower bound this expression over $\|u\|_2= 1$, it suffices to study a lower bound on 
\begin{align*}
   \underbrace{ \inf_{\|u\|_2 = 1} \sum_{\substack{t=0,\\ w_t^{(i)}=0}}^{T-1} |u^T \phi(x_t) |}_{(A)}  - \underbrace{\Biggr\|\sum_{\substack{t=0,\\ w_t^{(i)}\ne0}}^{T-1} \phi(x_t) \cdot \mathrm{sgn}(w_t^{(i)})\Biggr\|_2}_{(B)}.
\end{align*}
Moreover, an additional lower bound can be established by letting $\xi_t\sim \mathrm{Bernoulli}(2p)$ (see Scenario \ref{as:prob}) independently over $t$, while letting $\mathbb{P}(\mathrm{sgn}(w_t^{(i)})>0)=\mathbb{P}(\mathrm{sgn}(w_t^{(i)})<0)$ whenever the attack occurs (see Theorem 3, \cite{kim2024prevailing}). Thus, we will study the two terms under this sign-symmetric disturbance structure.

\medskip
\textbf{Term (A):} This term is related to the time index in the absence of attack. However, sufficient excitation  is inevitable to lower bound this term; thus, to take advantage of Assumption \ref{as:excitation_attack}, we  consider the time index set $\mathcal{T}_{non}:= \{t\in\{1,\dots, T-1\}: w_t^{(i)}= 0, ~\xi_{t-1}= 1\}$.  For each $t\in\mathcal{T}_{non}$, we have for a fixed $\|u\|_2=1$ that 
\begin{align}\label{lemma3}
    \mathbb{P}\Bigr(  |u^T \phi(x_t)| \ge \frac{\lambda}{2}\, \Bigr|\,\mathcal{F}_{t-1}  \Bigr) \ge \frac{\bar c \lambda^4}{ (\sqrt{n} L \sigma)^4},
\end{align}
where $\bar c >0$ is a constant. Note that the above relationship adapts Lemma 3 in \cite{zhang2024exact}; moreover, additional $n^2$ factor is introduced compared to \cite{zhang2024exact} since they assumed $\| \|w_t\|_2 \|_{\psi_2} \le \sigma$, while our standard assumption on sub-Gaussian disturbances implies $\| \|w_t\|_2 \|_{\psi_2} \le \sqrt{n}\sigma$ (see Remark \ref{subgnorm}). Define $Y_t := \frac{\lambda}{2}\mathbb{I}\{|u^T \phi(x_t)| \ge \frac{\lambda}{2}\}$, where $\mathbb{I}\{\cdot\}$ denotes the indicator function, which satisfies $|u^T \phi(x_t)| \ge Y_t$ for all $t$.  
Since $Y_t-\mathbb{E}[Y_t\,|\,\mathcal{F}_{t-1}]$ forms a Martingale difference sequence,  we use 
the Chernoff bound on martingales to \eqref{lemma3} to obtain 
\begin{align*}
    \mathbb{P}\Bigr( \sum_{t\in \mathcal{T}_{non}}Y_t \le \frac{\bar c |T_{non}|\lambda^5}{4(\sqrt{n}L\sigma)^4} \Bigr)\le \exp\Bigr(- \frac{\bar c |T_{non}|\lambda^4}{8(\sqrt{n}L\sigma)^4}\Bigr).
\end{align*}
 Since $|u^T\phi(x_t)|\ge Y_t$, we have 
\begin{align*}
  &  |\mathcal{T}_{non}| = \Omega\Bigr( \frac{(\sqrt{n}L\sigma)^4}{\lambda^4}\log\Bigr(\frac{1}{\delta}\Bigr) \Bigr) ~~\Longrightarrow \\&\hspace{10mm}\mathbb{P}\Bigr(\sum_{t\in \mathcal{T}_{non}}|u^T\phi(x_t)|\ge \frac{\bar c |T_{non}|\lambda^5}{4(\sqrt{n}L\sigma)^4}\Bigr)\ge 1-\frac{\delta}{4}.\numberthis\label{termainter}
\end{align*}
Moreover, we have $\sum_{t\in |\mathcal{T}_{non}|} \|\phi(x_t)\|_2$ has $\psi_2$-norm of $|\mathcal{T}_{non}|(\frac{\sqrt{n}L \sigma}{1-\rho L})$. Given $\epsilon>0$, when $\|u-\tilde u\|_2\le \epsilon$,  
we have 
\begin{align*}
   \biggr| \sum_{t\in \mathcal{T}_{non}}|u^T\phi(x_t)&|-  \sum_{t\in \mathcal{T}_{non}}|\tilde u^T\phi(x_t)|\biggr|\\&=\mathcal{O}\Bigr( \epsilon |\mathcal{T}_{non}|\Bigr(\frac{\sqrt{n}L \sigma}{1-\rho L}\Bigr)\log\Bigr( \frac{1}{\delta}\Bigr)\Bigr)\numberthis\label{termainter2}
\end{align*}
with probability at least $1-\frac{\delta}{4}$. Let $\epsilon = \mathcal{O}\bigr(\frac{\lambda^5 (1-\rho L)}{(\sqrt{n}L\sigma)^5 \log(1/\delta)}  \bigr)$ to bound \eqref{termainter2} by $\frac{\bar c |\mathcal{T}_{non}|\lambda^5}{8(\sqrt{n}L\sigma)^4}$. By covering number arguments (see   Corollary 4.2.11, \cite{vershynin2025high}), we can construct an $\epsilon$-net of at most  $(1+\frac{2}{\epsilon})^m$ for the vectors $u$ that simultaneously satisfy $\sum_{t\in \mathcal{T}_{non}}|u^T\phi(x_t)|\ge \frac{\bar c |\mathcal{T}_{non}|\lambda^5}{4(\sqrt{n}L\sigma)^4}$ with probability at least $1-\frac{\delta}{4}$, which is attained by replacing $\delta$ in \eqref{termainter} with $\frac{\delta}{(1+\frac{2}{\epsilon})^m}$, which requires 
\begin{align}\label{tnon}
    |\mathcal{T}_{non}| = \Omega\Bigr(\frac{(\sqrt{n}L\sigma)^4}{\lambda^4}\Bigr[m\log\Bigr( \frac{nL\sigma }{\lambda(1-\rho L)}\Bigr) + \log\Bigr(\frac{1}{\delta}\Bigr) \Bigr]\Bigr).
\end{align}
Under \eqref{tnon}, we can take a union bound over this net and  \eqref{termainter2} to guarantee \begin{align}\label{tnonf}
 \text{Term (A)} \ge  \inf_{\|u\|_2=1}  \sum_{t\in \mathcal{T}_{non}}|u^T\phi(x_t)|\ge \frac{\bar c |\mathcal{T}_{non}|\lambda^5}{8(\sqrt{n}L\sigma)^4}
\end{align}
with probability at least $1-\frac{\delta}{2}$. 

\medskip
\textbf{Term (B):} This term is related to the time index under attack. Define $\mathcal{T}_{att}:= \{t\in\{0,\dots, T-1\}: w_t^{(i)}\ne 0\}$. 
Bounding this term is similar to the approach to bound Term (B) in the proof of Lemma \ref{secondterm}. Since $|\mathrm{sgn}(w_t^{(i)})|\le 1$, 
 we have 
    \begin{align*}
  &\bigr\|  \|\mathrm{sgn}(w_t^{(i)}) \phi(x_t)\|_2~\bigr|~\mathcal{F}_t\bigr\|_{\psi_2} =\mathcal{O} \bigr( \|\phi(x_t)\|_2\bigr),
\end{align*}
Considering that 
$\mathbb{E}[\mathrm{sgn}(w_t^{(i)})] = 0$, we apply the same technique for bounding the sum of the stopped martingale difference sequence used in \eqref{stoppingtime}-\eqref{termB}, except that $\mu$ and $T$ are now replaced with $1$ and $|\mathcal{T}_{att}|$, respectively. Thus, we have 
\begin{align}\label{tattf}
    \text{Term (B)} = \mathcal{O}\Bigr(  \frac{\sqrt{|\mathcal{T}_{att}|} \sqrt{n}L\sigma}{1-\rho L}\log\Bigr(\frac{1}{\delta}\Bigr)   \Big)
\end{align}
with probability at least $1-\frac{\delta}{4}$. 

Using the Chernoff bound to $\{\xi_t\}_{t\ge 0}$, when $T=\Omega\bigr( \frac{1}{p(1-2p)}\log\bigr(\frac{1}{\delta}\bigr) \bigr)$, we have 
\begin{align}\label{tnontatt}    |\mathcal{T}_{non}| \ge p(1-2p)T, \quad |\mathcal{T}_{att}| \le 4pT
\end{align}
with probability at least $1-\frac{\delta}{4}$. We construct the union bound over \eqref{tnonf}, \eqref{tattf}, and \eqref{tnontatt} to have $\text{Term (A)}-\text{Term (B)} =\Omega(\frac{p(1-2p)T\lambda^5}{(\sqrt{n} L\sigma)^4})$ with probability at least $1-\delta$ when the time complexity satisfies both \eqref{tnon} and \begin{align}\label{timeaux}
     T=\Omega\biggr( \frac{(\sqrt{n}L\sigma)^{10}}{p(1-2p)^2\lambda^{10}  (1-\rho L)^2}\log^2\Bigr(\frac{1}{\delta}\Bigr)\biggr).
\end{align}

To ensure that the aforementioned argument holds for all $i\in\{1,\dots,n\}$, we substitute $\frac{\delta}{n}$ for $\delta$ in \eqref{tnon} and \eqref{timeaux} to obtain the time  \eqref{timeboundattack}. This completes the proof.
\end{proof}

\medskip
We now prove that the estimation error of the Huber estimator is bounded by $\mathcal{O}(\mu)$ under sparse nonzero-mean adversarial attacks.

\textit{\textbf{Proof of Theorem \ref{attackscenario}}}:
Let
$ f^{(i)}(a_i) := \sum_{t=0}^{T-1} |x_{t+1}^{(i)}-a_i^T \phi(x_t)|$ and   $h^{(i)}(a_i):= \sum_{t=0}^{T-1} H_\mu (x_{t+1}^{(i)}-a_i^T \phi(x_t) )$.

 Note that we have
    \begin{align*}
     \mu |z|-  H_\mu(z)  = \begin{dcases*}
         \mu|z| - \frac{1}{2}z^2 & if $|z|\le \mu$,\\ \frac{1}{2}\mu^2 & if $|z| > \mu$,
     \end{dcases*} 
    \end{align*}
    where $0\le \mu|z| - \frac{1}{2}z^2 \le \frac{1}{2}\mu^2$ if $|z| \le \mu$. This implies that 
    \begin{align*}
     0\le    \mu f^{(i)}(a_i) - h^{(i)}(a_i) \le \frac{\mu^2 T}{2}, \quad \forall a_i\in \mathbb{R}^m.
    \end{align*}
    Then, for every $i\in\{1,\dots,n\}$,  we obtain the relationship
    \begin{align*}
         \mu f^{(i)}(\hat a_i) \le h^{(i)}(\hat a_i) +  \frac{\mu^2 T}{2} &\le h^{(i)}(\bar a_i) +  \frac{\mu^2 T}{2}\\&\le \mu f^{(i)}(\bar a_i) +  \frac{\mu^2 T}{2},
    \end{align*}
where the second inequality follows from the optimality of $\hat a_i$ to \eqref{hubermin} for every $i$. This quantifies an upper bound on  $\mu f^{(i)}(\hat a_i)-\mu f^{(i)}(\bar a_i)$.  Lemma \ref{lem:excitation} implies that under \eqref{timeboundattack}, there exists a constant $c>0$ such that
\begin{align*}
  c\mu T &\frac{p(1-2p)\lambda^5 }{n^2 L^4\sigma^4 } \|\hat a_i - \bar a_i\|_2  \le \mu f^{(i)}(\hat a_i)-\mu f^{(i)}(\bar a_i) \le  \frac{\mu^2 T}{2}
\end{align*}
  for all $i$, with probability at least $1-\delta$. Rearranging the left-hand  and right-hand sides completes the proof.
\hfill $\blacksquare$

\section{Conclusion} \label{sec:conclusion}
In this paper, we introduce a robust system identification framework based on the Huber estimator,  a principled middle ground between the least-squares and the $\ell_1$-norm estimators. We prove that, given a positive noise density around zero, the Huber estimator achieves the optimal $\mathcal{O}(1/\sqrt{T})$ error rate when the disturbances are symmetric or the basis functions are linear.  Furthermore, we establish a bounded constant estimation error against sparse adversarial attacks.
This work provides the first unified theoretical guarantees for the Huber estimator across both extreme disturbance regimes.



\printbibliography

\end{document}